\documentclass[11pt]{amsart}
\usepackage{graphicx}
\usepackage{pinlabel}
\usepackage{color}

\theoremstyle{plain}
\newtheorem{thm}{Theorem}[section]

 \newtheorem{lem}[thm]{Lemma}
 \newtheorem{prop}[thm]{Proposition}
 
 \newtheorem{dfn}[thm]{Definition}

\newcommand{\diam}{\text{\rm diam}}

\begin{document}

\title{Alternating knots with large boundary slope diameter}

\author{Masaharu Ishikawa}
\address[M. Ishikawa]{Department of Mathematics, Hiyoshi Campus, Keio University, 4-1-1 Hiyoshi, Kohoku-ku, Yokohama-shi, Kanagawa, 223-8521, Japan}
\email{ishikawa@keio.jp}
\author{Thomas W.~Mattman}
\address[T. W. Mattman]{Department of Mathematics and Statistics,
California State University, Chico,
Chico, CA 95929-0525}
\email{TMattman@CSUChico.edu}
\author{Kazuya Namiki}
\address[K. Namiki]{Department of Mathematics, Saitama University, 255 Shimo-Okubo, Sakura-ku, Saitama-shi, Saitama, 338-8570, Japan}
\author{Koya Shimokawa} 
\address[K. Shimokawa]{Department of Mathematics, Saitama University, 255 Shimo-Okubo, Sakura-ku, Saitama-shi, Saitama, 338-8570, Japan}
\email{kshimoka@rimath.saitama-u.ac.jp}

\thanks{The first author is supported by JSPS KAKENHI Grant Numbers JP19K03499, JP17H06128 and Keio University Academic Development Funds for Individual Research.
The second author thanks the Math Department of Saitama University for their hospitality during 
multiple visits related to this project.
The last author is supported by JSPS KAKENHI Grant Numbers 16K13751, 16H03928, 17H06460, 17H06463, and JST CREST JPMJCR17J4.
}

\dedicatory{Dedicated to Steve Boyer on his 65th birthday.}

\begin{abstract}
We show that, for an alternating knot, the ratio of the diameter of the set of boundary slopes to the crossing number can be arbitrarily large.
\end{abstract}

\maketitle

\section{Introduction}

Let $K$ be a knot in $S^3$ and $c(K)$ the crossing number of $K$.
The {\em diameter} of the set of boundary slopes of $K$, denoted by $\diam(K)$,  is the difference, as rational numbers, between the maximal boundary slope and the minimal boundary slope of $K$.
We show that, for alternating knots, the ratio of the diameter of the boundary slopes to the crossing number can be arbitrarily large.

Curtis and Taylor ~\cite{CT} showed the diameter of the boundary slopes of essential spanning surfaces of an alternating knot $K$ is $2c(K)$.
See also \cite{Howie}.
Ichihara and Mizushima~\cite{IM} argued that, for a Montesinos knot $K$, the ratio $\frac{\diam(K)}{c(K)}$ is at most two, with equality if and only if $K$ is alternating.
In fact, they showed $\frac{\diam(K)}{c(K)}\ge 2$ in the case of alternating knots  (whether or not they are Montesinos).
According to Culler's calculation of the A-polynomial \cite{Culler},
the $8$--crossing non-Montesinos alternating knots $8_{17}$ and $8_{18}$ have boundary slopes
$-14$ and $14$ so that the diameter is at least $28$,
which is larger than twice the crossing number.
Those calculations are confirmed by Kabaya, who, 
by applying his method for the deformation variety~\cite{Kabaya-JKTR},
also demonstrated diameters in excess of twice the crossing 
number for the alternating knots $9_{32}$, $9_{33}$, $9_{34}$, $9_{39}$, $10_{82}$, and $10_{94}$. In the case
of $9_{39}$, he showed the diameter is at least $30$, which is more than three times the crossing number \cite{Kabaya}.
Dunfield and Garoufalidis \cite{DG} also constructed interesting examples of alternating knots with rational boundary slopes 
using spun-normal surfaces.

Here we prove that, for an alternating knot, the ratio of the diameter of the boundary slopes to the crossing number can be arbitrarily large.

\begin{thm}
\label{thm:diameter}
For any positive number $r$,
there exists an alternating knot $K$ such that the ratio
 $\frac{\diam(K)}{c(K)}>r$.
In particular, there is a sequence of alternating knots
$\{K_n\}_{n=1}^\infty$
such that ${\displaystyle \lim_{c(K_n)\to \infty}}\frac{\diam(K_n)}{c(K_n)}=\infty$.
\end{thm}

We adapt Hatcher and Oertel's~\cite{HO} method for Montesinos knots
to arborescent knots formed as the product of two Montesinos tangles.
In the next section we give a framework for such knots and in Section~\ref{sec:surf}
we describe how to construct candidate surfaces from edgepaths in this setting. 
Section~\ref{sec:bdy} shows how to calculate the corresponding boundary slopes. Finally, 
in Section~\ref{sec:thmdpf} we prove our main theorem.

\section{Arborescent knots}

In this article, a {\em tangle} $(B,T)$ is a pair of a 3-ball $B$ and its properly embedded 1-submanifold $T$ in $B$ with 4 boundary points NE, NW, SE and SW.
Sometimes a tangle is simply denoted by $T$.
We assume that $B$ is embedded in $\mathbb R^3$ and the four points NE, NW, SE and SW are on a plane $P$ in $\mathbb R^3$.
A {\em rational tangle} is a tangle that is homeomorphic to the pair $(D\times [0,1], \{p_1,p_2\}\times [0,1])$,
where $D$ is a disk and $p_1$ and $p_2$ are points in the interior of $D$.

\begin{dfn}
Let $(B_1,T_1)$ and $(B_2,T_2)$ be two tangles.
We identify the right hemisphere of $\partial B_1$ and the left hemisphere of $\partial B_2$ to form a tangle $(B_1\cup B_2, T_1\cup T_2)$.
The resulting tangle is denoted by $T_1+T_2$ and called the {\em tangle sum} of $T_1$ and $T_2$.
For rational tangles $(B_1,R_1), \cdots, (B_n,R_n)$, the tangle sum $R_1 + \cdots + R_n$ is called a {\em Montesinos tangle}.
Let $\Delta_i$ be the disk $B_i\cap B_{i+1}$.
Let $A$ be $\partial \Delta_i$.
The loop $A$ is called the axis of the Montesinos tangle.
\end{dfn}

\begin{dfn}
Let $R$ and $T$ be tangles. Let $R'$ denote the tangle that results 
from $\frac{\pi}{2}$-rotation of the reflection of $R$. 
Here the reflection means the tangle obtained by exchanging ``over'' and ``under'' at all crossings in the tangle.
Then the 
{\em tangle product} of $R$ and $T$, denoted $R \circ T$ is the sum
$R'+T$. See Figure~\ref{Fig:tanglesum}.
We also define the axis of $R \circ T$ as the union of the boundaries of 
the left and right hemispheres of the attaching tangles.
\end{dfn}

\begin{figure}[htb]
\begin{center}
\includegraphics[width=10cm, bb=129 622 471 712]{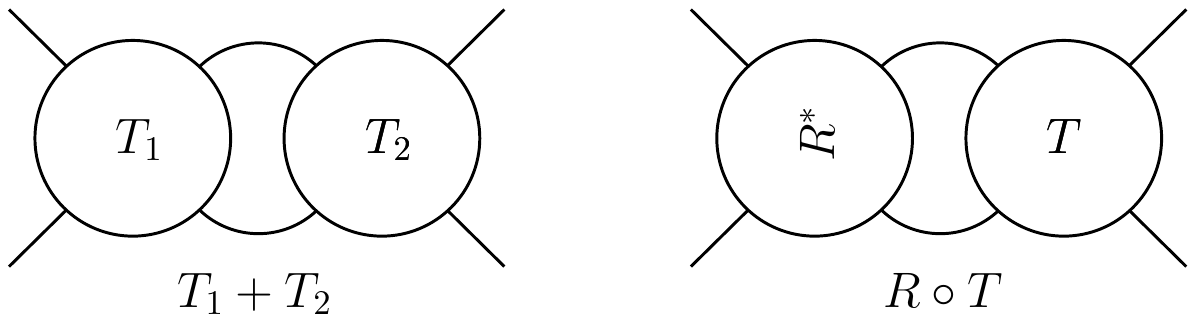}
\end{center}
\caption{\label{Fig:tanglesum}%
Tangle sum and product. $R^*$ denotes the reflection of $R$.
$R'$ is the tangle that results from $\frac{\pi}{2}$-rotation of $R^*$ .}
\end{figure}

\begin{dfn}
Let $(B,T)$ be a tangle.
We can construct a knot or a link by connecting NE and NW with an arc in $P\cap \partial B$
and SE and SW with an arc in $P\cap \partial B$.
This knot or link is denoted by $N(T)$ and called the {\em numerator} of $T$.
\end{dfn}

\begin{dfn}
A knot $K$ is an {\em SN knot} if there are Montesinos tangles $T_1$ and $T_2$ such that
 $K=N(T_1\circ T_2)$.
We call $T_2$ the {\em outer tangle}. The tangle obtained from $T_1$ by 
$\frac{\pi}{2}$-rotation of the reflection will be denoted $T_1'$ and called the 
{\em inner tangle}. 
\end{dfn}

Note that for an alternating Montesinos tangle $T$, the SN knot $N(T \circ T)$ is also alternating.

\begin{dfn}
A tangle is called an {\em algebraic tangle} if it is obtained by a finite sequence of tangle sums and tangle products of rational tangles.
The numerator $N(T)$ of an algebraic tangle $T$ is called an {\em arborescent knot or link}.
\end{dfn}

Note that the numerator of $(B,T)$ can be obtained by identifying the left hemisphere of $B$ with the right one.
If an algebraic tangle $T$ is made of rational tangles $(B_1,R_1),\cdots,(B_n,R_n)$,
then the arborescent knot $N(T)$ is in the $S^3$ that is the union of the $3$-balls $B_1\cup \cdots \cup B_n$.
The union of all the axes is an embedded graph in $S^3$.

\section{
\label{sec:surf}%
Candidate surfaces and edgepaths}

In this section we describe how to build candidate surfaces from edgepaths in the context of 
an arborescent knot.
We closely follow the discussion of Hatcher and Oertel~\cite{HO} 
and begin with a review of some terminology from that 
paper. In Figure~\ref{fig:HOTT}, we show the train track with weight $(a,b,c)$. 
The triple $(a,b,c)$ also labels a point of the  Diagram $\mathcal D$ in \cite{HO}
with vertical coordinate 
``slope'' $v=\frac{c}{a+b}$ and horizontal coordinate $u=\frac{b}{a+b}$.
Let $\langle\frac{p}{q}\rangle$ denote the point $(1,q-1,p)$.

Given such a train track, Hatcher and Oertel describe how to construct an edgepath $\gamma_i$ with the following four properties:
\begin{itemize}
\item[(E1)] The 
starting
point of $\gamma_i$ lies on the edge $\langle \frac{p_i}{q_i}, \frac{p_i}{q_i}\rangle$, and if the starting point is not the vertex $\langle \frac{p_i}{q_i}\rangle$, then the edgepath $\gamma_i$ is constant.
\item[(E2)] $\gamma_i$ is minimal, i.e., it never stops and retraces itself, nor does it go along two sides of the same triangle of $\mathcal D$ in succession.
\item[(E3)] The ending points of the $\gamma_i$'s are all rational points of $\mathcal D$ which
all lie in one vertical line and whose vertical coordinates add up to zero.
\item[(E4)] $\gamma_i$ proceeds monotonically from right to left, ``monotonically" in a weak sense that motion along vertical edges is permitted.
\end{itemize}

As in~\cite{HO}, from an edgepath $\gamma_i$ we can construct a surface $S_i$
in $B_i$ with $\partial S_i\subset \partial B_i\cup K$, 
which we will call 
the {\it surface $S_i$ given by the edgepath $\gamma_i$}. 
So far, except for property E3, we have
described Hatcher and Oertel's construction for an individual rational tangle $(B_i,R_i)$.

\begin{figure}[htb]
\begin{center}
\includegraphics[width=11cm, bb=129 538 484 713]{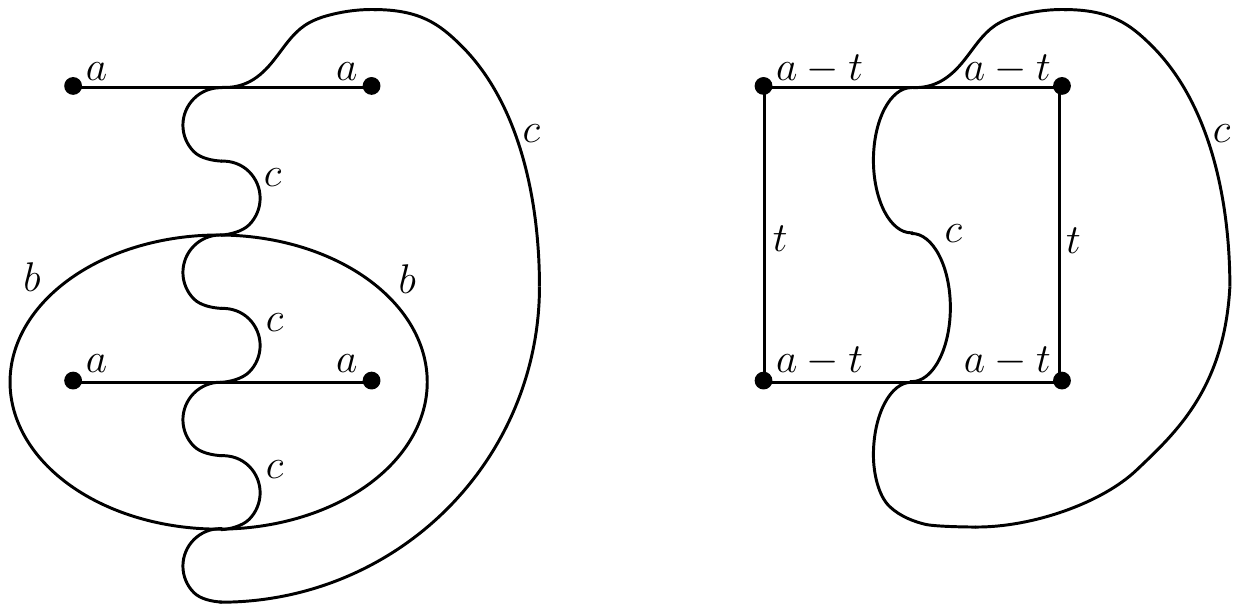}
\end{center}
\caption{
\label{fig:HOTT}%
A train track with weight $(a,b,c)$. The right figure has $t$ $\infty$-edges.}
\end{figure}

Now, let $K$ be an arborescent knot made from rational tangles $(B_1,R_1),\cdots,(B_n,R_n)$.
We will construct a properly embedded surface in the exterior of $K$ using the $S_i$'s
given by the edgepaths $\gamma_i$.
As an arborescent 
knot 
can be constructed by a sequence of tangle sums and tangle products,
we take these operations in turn, starting with the sum.

\begin{lem}[\cite{HO}]\label{lem:sum}
Let $T_1$ and $T_2$ be tangles with 
surfaces $\hat S_1$ and $\hat S_2$ 
having train tracks with weights $(a_1,b_1,c_1)$ and $(a_2,b_2,c_2)$ on the boundary,
respectively.
Suppose $a_1=a_2$ and $b_1=b_2$.
Then we can construct a surface $S$ in $T_1+T_2$ for the triple $(a_1,b_1,c_1+c_2)$
by gluing $\hat S_1$ and $\hat S_2$ canonically.
\end{lem}

Next we consider the tangle product.
The tangle product can be obtained by a sequence of operations consisting of $\frac{\pi}{2}$-rotation, reflection, and tangle sum.
In the next lemma, we will consider how the weight of the train track is changed by $\frac{\pi}{2}$-rotation and reflection.

\begin{lem}\label{lem:outer}
Let $T$ be a tangle 
with surface having a train track with weight $(a,b,c)$ on the boundary.
Then the weight $(x,y,z)$ of the train track of the tangle $T'$ obtained from $T$ by $\frac{\pi}{2}$-rotation and reflection satisfies the following in the case $c>0$ (respectively, $c<0$).
\begin{enumerate}
\item If $T$ has neither slope $0$ edges nor slope $\infty$ edges,
then $(x,y,z)=(a,c-a,a+b)$ (resp.\ $(a, |c|-a, -(a+b))$). 
\item If $T$ has slope 0 edges and does not have slope $\infty$ edges,
then, $T'$ has $a-|c|$ slope $\infty$ edges and triple $(x,y,z) = (c,0,b+c)$ (resp., $(|c|, 0, -(b+ |c|))$).
\item If $T$ does not have slope $0$ edges and has $t$ slope $\infty$ edges where $t<a$, 
then $(x,y,z) = (a, c-a+t,a-t)$ (resp.\ $(a, |c|-a+t,t-a)$).
\item If $T$ has both slope $0$ edges and $t$ slope $\infty$ edges, where $t<a-|c|$,
then $T'$ has $a-t-|c|$ slope $\infty$ edges and triple $(x,y,z)=(c+t,0,c)$ (resp.\ $(|c|+t,0,-c)$).
\end{enumerate}
\end{lem}

\begin{proof}
We show the assertion in Case (1). In this case, the train track for $T$ is the one shown on the left in Figure~\ref{fig:HOTT}.
If $c>0$ then the $\frac{\pi}{2}$-rotation of its reflection becomes the train track on the
top left in Figure~\ref{fig:iso1}. Then the assertion follows from the isotopy shown in Figure~\ref{fig:iso1}. 
If $c<0$, the isotopy is given by the mirror image of the isotopy in Figure~\ref{fig:iso1}.

The assertions for Cases (2), (3) and (4) follow from the isotopies shown in Figures~\ref{fig:iso2}, \ref{fig:iso3} and \ref{fig:iso4}, respectively.
\end{proof}

\begin{figure}[htb]
\begin{center}
\includegraphics[width=12cm, bb=129 354 536 713]{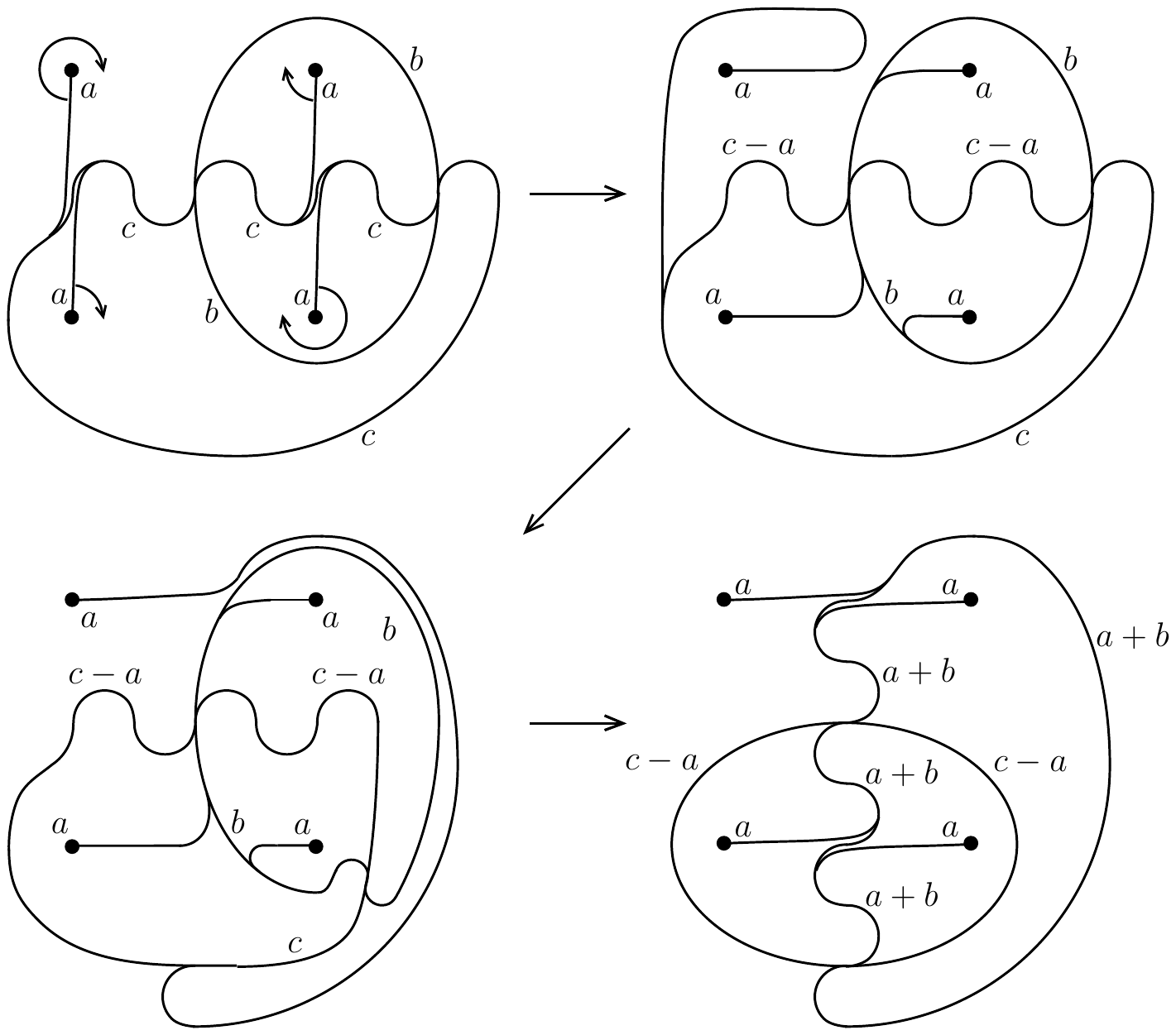}
\end{center}
\caption{
\label{fig:iso1}%
Case (1) : 
$T$ has neither slope $0$ edges nor slope $\infty$ edges.}
\end{figure}

\begin{figure}[htb]
\begin{center}
\includegraphics[width=16cm, bb=129 578 594 713]{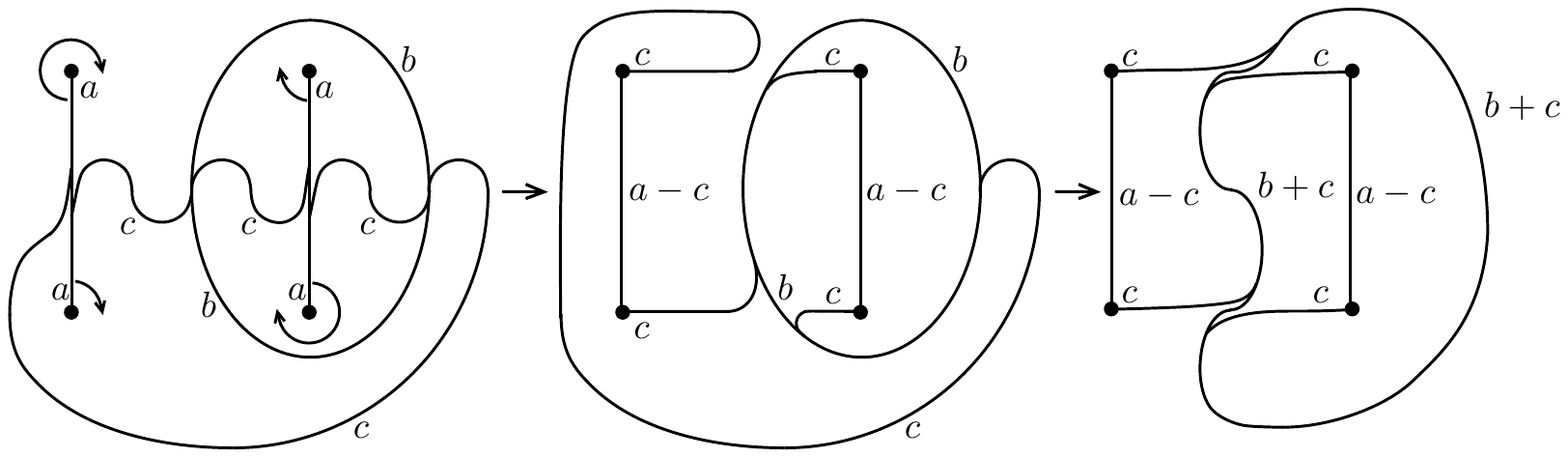}
\end{center}
\caption{
\label{fig:iso2}%
Case (2) : 
$T$ has slope $0$-edges. Then $T'$ has $a-c$ slope $\infty$ edges.}
\end{figure}

\begin{figure}[htb]
\begin{center}
\includegraphics[width=12cm, bb=129 566 470 713]{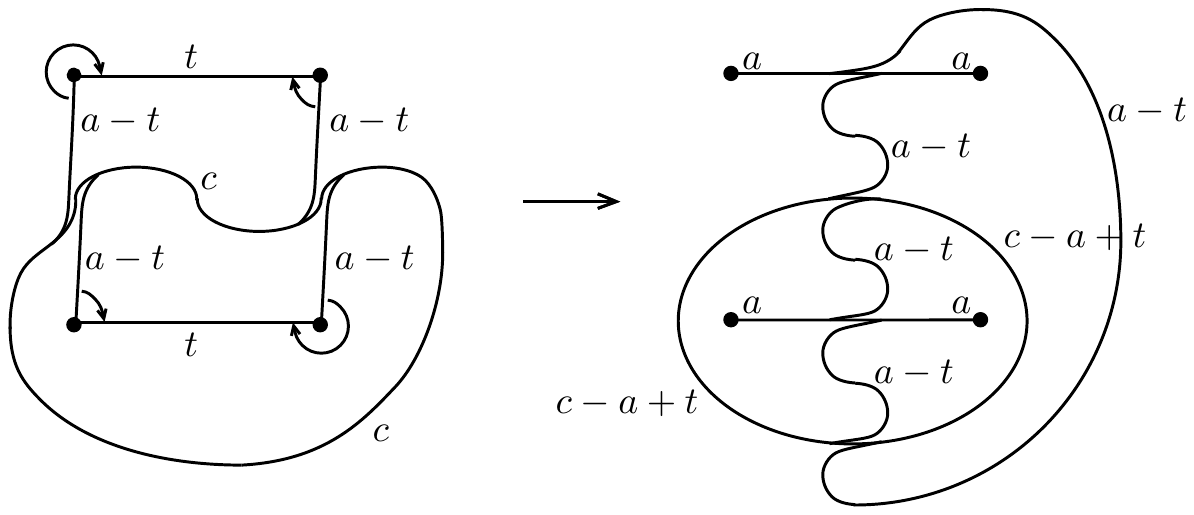}
\end{center}
\caption{
\label{fig:iso3}%
Case (3) : 
$T$ has $t$ slope $\infty$ edges with $t < a$.}
\end{figure}

\begin{figure}[htb]
\begin{center}
\includegraphics[width=12cm, bb=129 577 459 713]{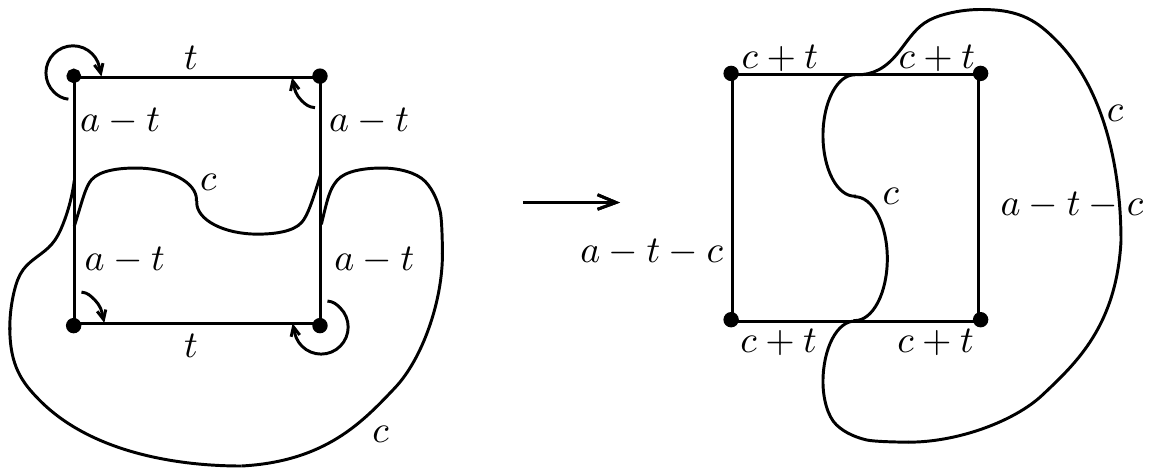}
\end{center}
\caption{
\label{fig:iso4}%
Case (4) : 
$T$ has $t$ slope $\infty$ edges with $t < a-c$. Then $T'$ has $a-t-c$ slope $\infty$ edges.}
\end{figure}

To construct a candidate surface 
of an arborescent knot
$K$, we replace property E3 of Hatcher and Oertel with the following.

\begin{itemize}
\item[(E3')] The ending points of the $\gamma_i$'s are all rational points of $\mathcal D$
chosen so that the surfaces $S_i$ given by the edgepaths $\gamma_i$
constitute a surface bounded by the knot $K$. 
We can easily check if the $S_i$ constitute a surface using Lemmas~\ref{lem:sum} and \ref{lem:outer}.
\end{itemize}


Note that, in~\cite{Wu}, Wu studied exceptional surgeries of large arborescent knots. 
The class of SN knots in our paper includes knots of type II in~\cite{Wu}.
In that paper, he used a different train track that 
works especially
well for knots of type II.
In our paper, we follow the notation of Hatcher and Oertel to emphasize that the observations in Lemmas~\ref{lem:sum} and~\ref{lem:outer} work for any arborescent knot.
In fact, these lemmas work for any tangle with a train track.

\section{
\label{sec:bdy}%
Calculation of boundary slopes}

In this section we will calculate the boundary slope of a candidate surface 
$S$ of an arborescent knot.
Let $(B_1,R_1),\cdots,(B_n,R_n)$ be rational tangles constituting an arborescent knot.
For each surface $S_i$ in $B_i$ given by an edgepath,
we define $\tau(S_i)$ to be $\tau(S_i)=2(e_--e_+)$, where $e_+$ (resp. $e_-$) is the number of edges of the edgepath which increase (resp. decrease) slope, allowing fractional values for $e_{\pm}$ corresponding to a final edge traversing only a fraction of an edge of the diagram $\mathcal D$.
See~\cite[page 460]{HO} for
a 
precise explanation.
This number $\tau(S_i)$ counts how many times the surface $S_i$ rotates around the knot in $B_i$. 
In particular, if  $S_i$ is given by a constant edgepath then we have $\tau(S_i)=0$.
In~\cite{HO}, $\tau(S)$ is defined to be the sum of these numbers for rational tangles constituting a Montesinos knot. We use the same idea for arborescent knots.

Let $\hat S$ be a surface in an algebraic tangle $T$ obtained by gluing surfaces given by edgepaths in the rational tangles constituting $T$.
We call $\hat S$ a {\it surface given by edgepaths}.
We define $\tau(\hat S)$ by applying Definitions~\ref{def41} and~\ref{def42} below inductively.

The first definition is for a tangle sum.

\begin{dfn}\label{def41}
Let $T_1$ and $T_2$ be algebraic tangles with surfaces $\hat S_1$ and $\hat S_2$, given by edgepaths, with the triples $(a,b,c_1)$ and $(a,b,c_2)$, respectively.
Let $\hat S$ be the surface in $T_1+T_2$ with the triple $(a_1,b_1,c_1+c_2)$ obtained by gluing $\hat S_1$ and $\hat S_2$, whose existence is stated in Lemma~\ref{lem:sum}.
Then $\tau(\hat S)$ is defined by
$\tau(\hat S)=\tau(\hat S_1)+\tau(\hat S_2)$.
\end{dfn}


Before giving the second definition, we define $\tau'(\hat S)$,
which 
measures how much the surface rotates along the knot during the isotopies in Figures~\ref{fig:iso1} through~\ref{fig:iso4}.
Let $T$ be a tangle with a candidate surface 
$\hat S$ 
having triple $(a,b,c)$ and 
$\hat S'$
be the surface obtained from 
$\hat S$
by one of the isotopies shown in Figures~\ref{fig:iso1} through~\ref{fig:iso4}.
Set 
$m$
to be the number of sheets that rotate around the knot during the isotopy.
Note that 
$m=a$
in Case~(1), 
$m=c$
in Cases~(2) and~(4) and 
$m=a-t$
in Case~(3), where $t$ is the number of slope $\infty$ edges in $T$.
We define 
$\tau'(\hat S)$
as follows.
\[
\tau'(\hat S)=
\begin{cases}
-\frac{2m}{a}\quad \text{if}\quad c>0\\
\frac{2m}{a}\quad \text{if}\quad c<0.
\end{cases}
\]


%

\begin{dfn}\label{def42}
Let $T_1$ and $T_2$ be algebraic tangles with surfaces $\hat S_1$ and $\hat S_2$ given by edgepaths, respectively.
Suppose that  $\hat S_1$ and $\hat S_2$ are glued in $T_1\circ T_2$ canonically as in Lemma~\ref{lem:sum} after the $\frac{\pi}{2}$-rotation and reflection of $T_1$ for the tangle product. 
Let $\hat S$ denote this surface in $T_1\circ T_2$.
Then $\tau(\hat S)$ is defined by 
$\tau(\hat S)=-\tau(\hat S_1)+\tau'(\hat S_1)+\tau(\hat S_2)$.
\end{dfn}



The value $\tau$ for a candidate surface of an arborescent knot is then defined as follows.

\begin{dfn}
Suppose that an arborescent knot $K$ is the numerator of an algebraic tangle $T$, i.e., $N(T)=K$,
and let $S$ be a candidate surface of $K$ obtained from a surface in $T$ given by edgepaths.
Then $\tau(S)$ is defined by $\tau(S)=\tau(\hat S)$.
\end{dfn}
 
To get the boundary slope of the candidate surface $S$ from $\tau(S)$, 
we need to determine $\tau(S_0)$ for a Seifert surface $S_0$.
Remark that we use the same tangle decomposition of $K$ when we calculate $\tau(S)$ and $\tau(S_0)$.

\begin{lem} 
\label{lem:tauS0P}
Let $K$ be an arborescent knot consisting of Montesinos tangles including a rational tangle with even denominator and let $S_0$ be a Seifert surface of $K$.
Then $\tau(S_0)$ is the sum of $(-1)^{k_i}\tau(S_i)$ for all Montesinos tangles $T_i$ constituting $K$, where $k_i$ is the number of reflections applied to the tangle $T_i$.
\end{lem}

\begin{proof}
As in \cite{HO}, we find a piece of the Seifert surface for each Montesinos tangle $(B_i, T_i)$.
In particular, the slopes of 
the surfaces in these tangles are $\infty$.
When we make a tangle product, we take the reflection of $B_i$ and rotate it by $\frac{\pi}{2}$.
Therefore, its slope becomes $0$ after the rotation. 
To glue this with other pieces, we add a saddle near $\partial B_i$ that changes the slope 
from $0$ to $\infty$.
We can find suitable edgepaths whose surface can be glued with keeping the orientability when each Montesinos tangle contains a rational tangle with even denominator, see the explanation in~\cite[page 461]{HO}.
The surface so constructed is single sheeted and orientable, hence a Seifert surface. 
Since adding a saddle does not change the number of twists around the knot, the contribution to $\tau(S_0)$ of this adjustment is $0$. Hence the assertion follows.
\end{proof}

\begin{lem}\label{lemmaSS0}
Let $S$ be a candidate surface of an arborescent knot. Then $\tau(S)-\tau(S_0)$ is the boundary slope of $S$.
\end{lem}

\begin{proof}
As in \cite{HO}, the value of $\tau$ can be determined by checking how the surface rotates
around the knot. 
This rotation number is additive for a tangle sum. 
Hence the rotation number of the surface in $T_1+T_2$ 
obtained from surfaces $\hat S_1$ and $\hat S_2$ in $T_1$ and $T_2$, respectively, given by edgepaths, 
as mentioned in Lemma~\ref{lem:sum}, is 
$\tau(\hat S_1)+\tau(\hat S_2)$. This is formulated in Definition~\ref{def41}.
Next we observe the tangle product $T_1\circ T_2$ with surfaces $\hat S_1$ and $\hat S_2$ given by edgepaths.
Remember that the tangle product can be obtained by a sequence of operations consisting of $\frac{\pi}{2}$-rotation,  reflection of $T_1$, and tangle sum.
Since $\hat S'_1$ is the reflection of $\hat S_1$,
the contribution of the twists of $\hat S_1$ to $\hat S'_1$ is $-\tau(\hat S_1)$.
The contribution during the isotopy in Figures~\ref{fig:iso1} through \ref{fig:iso4}
is the value $\tau'(\hat S_1)$, which can be verified directly from those figures.
Therefore,  the rotation number becomes $-\tau(\hat S_1)+\tau'(\hat S_1)+\tau(\hat S_2)$ as in Definition~\ref{def42}. Thus $\tau(S)$ counts the rotation of $S$ around the knot correctly.
The boundary slope of $S$ is given by the difference $\tau(S)-\tau(S_0)$ as in~\cite{HO}.
\end{proof}

\section{
\label{sec:thmdpf}%
Proof of Theorem~\ref{thm:diameter}}

For $n \geq 2$, let $K_n$ denote the SN knot $N(( - \frac{1}{n}+\frac{1}{n+1})\circ ( - \frac{1}{n}+\frac{1}{n+1}))$, see Figure~\ref{fig:KNSN}. Note that $K_n$ is achiral.
Here 
$K_n=N(T_1\circ T_2)$ 
and $(B_1,T_1)=(B_{1,1},T_{1,1})+(B_{1,2},T_{1,2})$ and  $(B_2,T_2)=(B_{2,1},T_{2,1})+(B_{2,2},T_{2,2})$.
Let $A_i$ be the {\em axis} defined by $\partial (B_{i,1}\cap B_{i,2})$ ($i=1,2$).

\begin{figure}[htb]
\begin{center}
\includegraphics[width=12cm, bb=129 405 554 713]{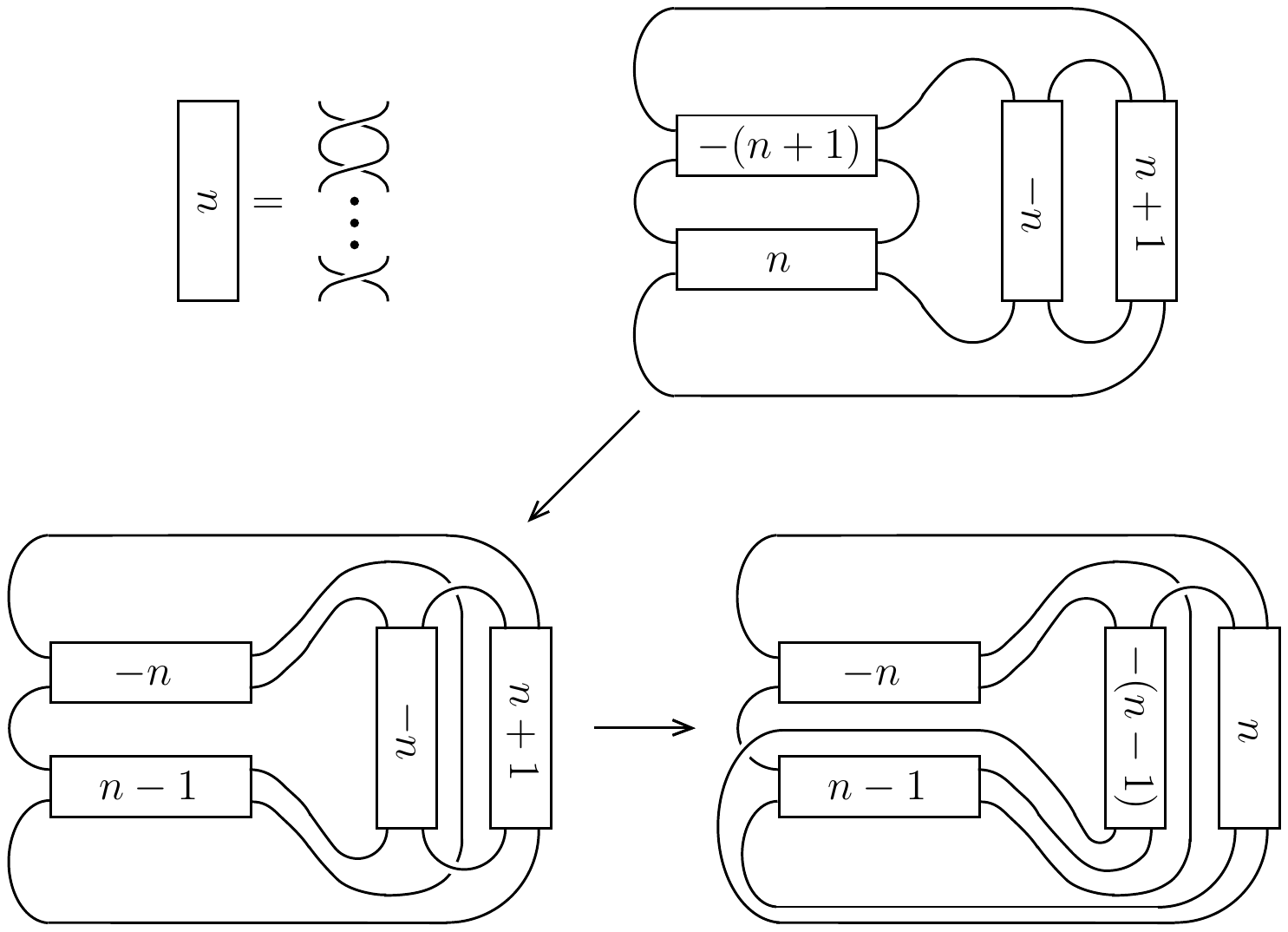}
\end{center}
\caption{
\label{fig:KNSN}%
$K_n$ is an alternating knot with crossing number $4n$.
The number $n$ in a box means the right-handed $n$ half twists and  $-n$ in a box means the left-handed $n$ half twists.}
\end{figure}

\begin{prop}
\label{prop:Knbdys}
$-2(n+1)^2+4$ and $2(n+1)^2-4$ are boundary slopes of essential spanning surfaces of $K_n$.
\end{prop}

\begin{proof}
Since $K_n$ is achiral, it is enough to show that $-2(n+1)^2+4$ is the boundary slope of an essential spanning surface.
For this, we describe the surface in terms of 
edgepaths in $T_1 = T_2 = - \frac1n + \frac{1}{n+1}$. In $T_1$, we use constant edgepaths in both the $-\frac1n$ and $\frac{1}{n+1}$ rational tangles.
Let the triples for the two rational tangles be $(1,n^2+n-1,-n-1)$ and $(1,n^2+n-1,n)$ so that $(1,n^2+n-1,-1)$ is the 
triple for $T_1$.
As there are neither $0$-edges nor $\infty$-edges, 
by Lemma~\ref{lem:outer},
the triple of $T_1$ after reflection, $\frac{\pi}{2}$-rotation and isotopy is $(1, 0, -n^2-n)$. 
Therefore the triple of $T_2$ must be $(1,0,n^2+n)$. 
We choose the edgepaths of the $-\frac1n$ and $\frac{1}{n+1}$ tangles in $T_2$ as $\langle -\frac{1}{n} \rangle \to \langle 0\rangle$ and $\langle \frac{1}{n+1} \rangle \to \langle \frac{1}{n} \rangle \to \cdots \to \langle \frac{1}{2} \rangle \to \langle 1 \rangle \to \cdots \to \langle n^2+n \rangle $, respectively.

Let $S$ be the candidate surface obtained from these edgepaths.
We must argue that the candidate surface $S$ is incompressible.
We adapt the arguments of Hatcher and Oertel~\cite{HO} to this case.
Let $D$ be a compressing disk of $S$ in $E(K)$.
We may assume that $D$ meets $\partial B_i$ and $A_i$ transversely and misses the intersection points $A_i\cap A_j$.
Set $G=D\cap (\partial B_1\cup \partial B_2)$, which is a graph in $D$.
We assume that the number of components of $G$ is minimal among all compressing disks for $S$.

If $G$ contains a loop that does not meet vertices, then by a surgery along an innermost disk of $D$, we will have another compressing disk whose
graph has  fewer components, a contradiction.
Hence there is no loop without vertices in $G$.

We next argue that $D \cap B_1 \neq \emptyset$. 
For a contradiction, assume $D \cap B_1$ is empty.
Let $(B_{2,1},T_{2,1})$ and $(B_{2,2},T_{2,2})$ denote the $-\frac1n$ and $\frac{1}{n+1}$ tangles in $B_2$, respectively.

Suppose that edgemost disks
are contained in $B_{2,1}$.
Let $D'$ be an edgemost disk that is bounded by the union of an arc $\alpha$ on $B_{2,1}\cap B_{2,2}$ and an arc $\beta$ on $\partial D$, see Figure~\ref{fig:conmdisk1} (left).
Since the edgepath of $T_{2,1}$ is $\langle -\frac1n \rangle \to \langle 0\rangle$,
the piece $S\cap B_{2,1}$ of the candidate surface $S$ in $B_{2,1}$ is a band as shown on the right in Figure~\ref{fig:conmdisk1}.
Then the disk $D'$ either lies in the position shown in the figure or is bounded by $\alpha\cup \beta$ such that the endpoints of $\alpha$ lie on the same connected arc of $S\cap \partial B_{2,1}$. The latter case can be removed by isotopy of $D'$.
Consider the former case.
The assumption $D\cap B_1=\emptyset$ implies that the arc $\alpha$ does 
not
intersect the hemisphere $B_1\cap B_{2,1}$. Since the band $S\cap B_{2,1}$ is twisted $n$ times, 
there exists such a disk $D'$ if and only if $|n|\leq 1$. In particular, if $n\geq 2$ then $D'$ and $B_1$ should intersect.
This contradicts the assumption that $D\cap B_1$ is an empty set.

\begin{figure}[htb]
\begin{center}
\includegraphics[width=11cm, bb=129 581 432 713]{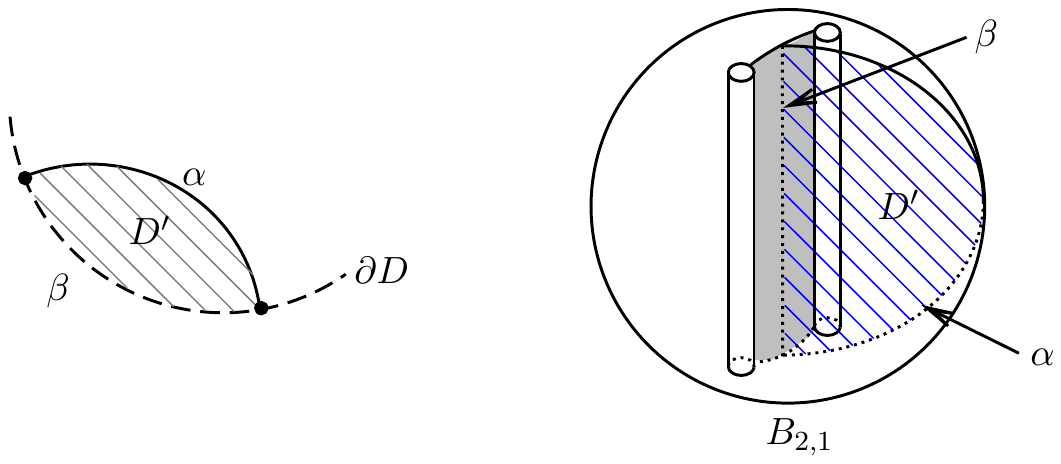}
\end{center}
\caption{
\label{fig:conmdisk1}%
The edgemost disk $D'$ in $B_{2,1}$.}
\end{figure}

Suppose that edgemost disks are in $B_{2,2}$.
Then there is a disk $\Delta$ in $D$ contained in $B_{2,1}$ whose boundary is
$\alpha_1, \beta_1, \ldots, \alpha_k, \beta_k$, where the $\alpha_i$ are on $B_{2,1}\cap B_{2,2}$ and the $\beta_i$ are contained in 
$\partial D$, see Figure~\ref{fig:conmdisk3}.
Since $D\cap B_1$ is empty, $\partial \Delta$ does not intersect the axis $\partial (B_{2,1}\cap B_{2,2})$.
However, there is no such disk $\Delta$ since 
$S\cap B_{2,1}$ is a band as mentioned before.

Therefore, in either case, we have a contradiction.

Now, applying the argument of \cite[Proposition 2.1]{HO}
to the tangles $B_{1,1}$, $B_{1,2}$ and $B_2$,
where $T_{1,1}$ and $T_{1,2}$
are the $-\frac1n$ and $\frac{1}{n+1}$ tangles in $B_1$, respectively,
$D \cap B_1$ includes at most one innermost disk $D'$, which means there is
at most one constant edgepath in $T_1$. However, by assumption, there are two. The contradiction shows there
can be no such compressing disk and $S$ is incompressible.

\begin{figure}[htb]
\begin{center}
\includegraphics[width=7cm, bb=182 599 372 712]{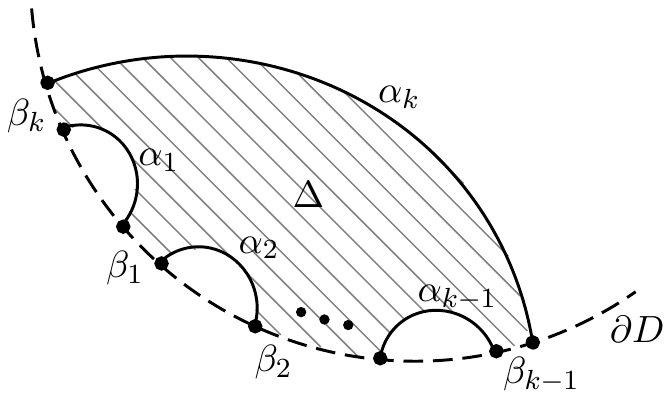}
\end{center}
\caption{
\label{fig:conmdisk3}%
The disk $\Delta$.}
\end{figure}

Finally, we calculate the slope associated to the incompressible surface $S$.
Since both of the tangles in $T_1$ have constant edgepaths, $\tau(S_1) = 0$. 
Since 
there are neither $0$ nor $\infty$ edges in the train track for $T_1$ and $c <0$, 
$\tau'(S_1)=2$.
For $\tau(S_2)$, there is one upward edge on $\mathcal D$ for the $-\frac1n$ tangle and there are
a 
 further $n^2+2n-1$ upward edges for the $\frac1{n+1}$ tangle, so $\tau(S_2) = 2( -1 - (n^2+2n-1)) = -2(n^2+2n)$.
The twist of the surface is therefore $\tau(S) = -\tau(S_1) +\tau'(S_1)+ \tau(S_2) 
= 0 + 2 - 2(n^2+2n) = -2(n+1)^2 + 4$
by Lemma~\ref{lemmaSS0}. 

Let $S_{0,1}$ and $S_{0,2}$ be pieces of a Seifert surface $S_0$ for $T_1$ and $T_2$, respectively, obtained as in~\cite{HO}.
Since $S_{0,1}$ and $S_{0,2}$ are given by the same edgepaths, we have $\tau(S_0) = - \tau(S_{0,1}) + \tau(S_{0,2}) = 0$ by Lemma~\ref{lem:tauS0P}.
Therefore the boundary slope of 
the incompressible surface $S$
is $\tau(S) - \tau(S_0) = -2(n+1)^2+4$
by Lemma~\ref{lemmaSS0}. 
\end{proof}

\begin{proof}[Proof of Theorem~\ref{thm:diameter}]
As in Figure~\ref{fig:KNSN}, $K_n$ has an irreducible alternating diagram with $4n$ crossings, so that $c(K_n) = 4n$.
By Proposition~\ref{prop:Knbdys},
$\diam(K_n) \geq 2(n+1)^2-4 - (-2(n+1)^2+4) = 4(n+1)^2-8$.
Therefore, $\frac{\diam(K_n)}{c(K_n)} \geq \frac{(n+1)^2-2}{n}$ tends to infinity, as required.
\end{proof}


\begin{thebibliography}{10}
\bibitem{Culler}
M. Culler, A table of A-polynomials computed via numerical methods. 
http://www.math.uic.edu/\~{}culler/\\
Apolynomials/

\bibitem{CT}
C.\ Curtis and S.\ Taylor,
The Jones polynomial and boundary slopes of alternating knots.
arXiv:0910.4912v3

\bibitem{DG}
N.\ Dunfield and S.\ Garoufalidis,
Incompressibity criteria for spun-normal surfaces.
Trans. Amer. Math. Soc. \textbf{364} (2012), 6109-6137.


\bibitem{HT}
A.\ Hatcher and W.\ Thurston,
Incompressible surfaces in 2-bridge knot complements.
Invent.\ Math.\  \textbf{79}  (1985), 225--246.

\bibitem{HO}
A. Hatcher and U. Oertel,
Boundary slopes for Montesinos knots.
Topology  \textbf{28}  (1989),  453--480.

\bibitem{Howie}
J.\ Howie, 
Boundary slopes of some non-Montesinos knots.
arXiv:1401.2726

\bibitem{IM}
K.\ Ichihara and S.\ Mizushima,  
Crossing number and diameter of boundary slope set of Montesinos  knot.
Comm.\ Anal.\ Geom.\ \textbf{16}  (2008), 565--589.
arXiv:math/0510370


\bibitem{Kabaya}
Y. Kabaya,
Private communication,
April, 2008.

\bibitem{Kabaya-JKTR}
Y. Kabaya,
A method to find ideal points from ideal triangulations.
J.\ Knot Theory Ramifications  \textbf{19}  (2010), 509--524.
arXiv:0706.0971.

\bibitem{Wu}
Y.-Q. Wu,
Exceptional Dehn surgery on large arborescent knots.
Pacific J. Math. {\bf 252} (2011), 219--243. 

\end{thebibliography}
\end{document}